\documentclass[12pt]{amsart}
\usepackage[margin=1in]{geometry}
\usepackage[english]{babel}
\usepackage[utf8]{inputenc}
\usepackage{subcaption}
\usepackage{amsmath}
\usepackage{amssymb}
\usepackage{amsfonts}
\usepackage{amsthm}
\usepackage{mathrsfs}
\usepackage{float}
\usepackage[all]{xy}
\usepackage[pdftex]{graphicx}
\usepackage{color}
\usepackage{cite}
\usepackage{url}
\usepackage{graphicx}
\usepackage{indent first}
\usepackage[labelfont=bf,labelsep=period,justification=raggedright]{caption}
\usepackage[english]{babel}
\usepackage[utf8]{inputenc}
\usepackage{hyperref}
\usepackage[colorinlistoftodos]{todonotes}
\usepackage{tkz-fct}
\usepackage{tikz}
\usetikzlibrary{calc}
\usepackage{multicol}
\PassOptionsToPackage{dvipsnames,svgnames}{xcolor}
\usepackage{textcomp}
\usepackage{algorithm,algpseudocode}

\DeclareMathOperator{\rad}{rad}

\def\multiset#1#2{\ensuremath{\left(\kern-.3em\left(\genfrac{}{}{0pt}{}{#1}{#2}\right)\kern-.3em\right)}}

\theoremstyle{plain}
\newtheorem{thm}{Theorem}
\newtheorem{lemma}[thm]{Lemma}
\newtheorem{cor}[thm]{Corollary}
\newtheorem{coror}[thm]{Cororollary}

\theoremstyle{definition}
\newtheorem{defn}[thm]{Definition}

\theoremstyle{remark}
\newtheorem{rem}[thm]{Remark}
\newtheorem{ex}{Example}

\numberwithin{equation}{section}
\numberwithin{thm}{section}

\begin{document}
\title[Fast Modular Exponentiation]{An Elementary Method For Fast Modular Exponentiation With Factored Modulus}
\author{Anay Aggarwal, Manu Isaacs}
\email{anay.aggarwal.2007@gmail.com, manu.isaacs@gmail.com}
\date{\today}



    








\maketitle
\begin{abstract}
    We present a fast algorithm for modular exponentiation when the factorization of the modulus is known. Let $a,n,m$ be positive integers and suppose $m$ factors canonically as $\prod_{i=1}^k p_i^{e_i}$. Choose integer parameters $t_i\in [1, e_i]$ for $1\le i\le k$. Then we can compute the modular exponentiation $a^n\pmod{m}$ in $O(\max(e_i/t_i)+\sum_{i=1}^k t_i\log p_i)$ steps (i.e., modular operations). We go on to analyze this algorithm mathematically and programmatically, showing significant asymptotic improvement in specific cases. Specifically, for an infinite family of $m$ we achieve a complexity of $O(\sqrt{\log m})$ steps, much faster than the Repeated Squaring Algorithm, which has complexity $O(\log m)$. Additionally, we extend our algorithm to matrices and hence general linear recurrences. The complexity is similar; with the same setup we can exponentiate matrices in $GL_d(\mathbb{Z}/m\mathbb{Z})$ in less than $O(\max(e_i/t_i)+d^2\sum_{i=1}^k t_i\log p_i)$ steps. This improves Fiduccia's algorithm and the results of Bostan and Mori in the case of $\mathbb{Z}/m\mathbb{Z}$. We prove analogous results for $\mathbb{Z}/p^k\mathbb{Z}$ ring extensions.
\end{abstract}
\section{Introduction}
An important problem at the intersection of cryptography and number theory is the Modular Exponentiation Problem. This is the problem of computing $a^n\pmod{m}$ for positive integers $a,n,m$. It happens that solving for $n$ in the equation $a^n\equiv b\pmod{m}$, called the Discrete Logarithm Problem (DLP), is computationally hard. This makes modular exponentiation useful in many cryptosystems, most notably the RSA public key cryptosystem and the Diffie-Hellman key exchange. The standard algorithm to solve the Modular Exponentiation Problem is the Repeated Squaring Algorithm, which runs in $O(\log n)$ steps (in this paper, a step is a modular multiplication). To our knowledge, there are no existing significant asymptotic improvements to this algorithm.
\newline \newline
In this paper, we present an efficient algorithm to solve the Modular Exponentiation Problem when the factorization of the modulus is known. Such a situation is potentially practically useful, as modular exponentiation is generally performed for encryption and thus the user has the freedom of choosing $m$ and knowings its factorization. In the 1984 paper \cite{Bach:CSD-84-186}, Bach proves the existence of a Hensel lifting type algorithm that can lift solutions to the discrete logarithm from modulo $p$ to modulo $p^e$ in polynomial time. One can intuit that in a similar way, there exists a speedup to modular exponentiation when the factorization of the modulus has some prime power in it; this is what our algorithm provides. The algorithm is quite elementary: it hinges on the binomial theorem and the clever recursive computation of inverses and binomial coefficients.  It also depends on a set of parameters: If $m$ can be factored as $\prod_{i=1}^k p_i^{e_i}$, we choose for each $i$ an integer parameter $t_i\in [1,e_i]$. We then compute $a^n\pmod{m}$ in $O(\max(e_i/t_i)+\sum_{i=1}^k t_i\log p_i)$ steps (again, steps are modular multiplications). We additionally provide an $O(1)$ memory algorithm for memory-sensitive scenarios, where we define $O(1)$ memory to mean the memory of an integer modulo $m$, or $\log m$ bits. For general $m$ we make asymptotic improvements as measured by specific metrics. For a certain infinite family of $m$ we achieve a complexity of $O(\sqrt{\log m})$. Furthermore, our algorithm profits from the development of other fast algorithms. For example, if one were to discover a faster modular exponentiation algorithm, we could use this as the $\textsc{modExp}$ function in our algorithm. This, in turn, makes our algorithm faster than the default one.
\newline \newline
In section 2 of this paper, we lay out some preliminary results that our algorithm and analysis require. In section 3, we present the algorithm with pseudocode. In section 4, we mathematically analyze our algorithm. Because its complexity depends on number-theoretic properties of $m$, we require results from analytic number theory to estimate the ``average" complexity. We discuss families of $m$ for which we make significant improvements. Then, in section 5, we present an version of our algorithm for matrices, linear recurrences, and $\mathbb{Z}/p^k\mathbb{Z}$ ring extensions. Finally, in section 6, we test our algorithm for prime powers against Python's built-in $\mathrm{pow}$ function to show our algorithm does practically. A Python implementation of the algorithm and programmatical analysis for it can be found at \cite{isaacs24}.
\section{Preliminaries}
The complexity of our algorithm depends on the number-theoretic properties of the modulus, so we will need multiple preliminary definitions to continue with this analysis. In general, we use the variable $p$ for a prime. We denote the canonical decomposition $m$ as $m = \prod p_i^{e_i}$. By convention, we let $a\pmod{b}$ denote the least residue of $a$ modulo $b$. We let $\nu_p(n)$ denote the $p$-adic valuation of $n$, and we use this notation interchangeably with $\nu(n,p)$. We let $\zeta(\bullet)$ be the Riemann Zeta Function and let $\varphi(\bullet)$ denote Euler's Totient Function. Additionally, we let $\vartheta(\bullet)$ denote Chebyshev's Function. We must define the radical of an integer $n$ because this is intimately related to the optimal complexity that our algorithm may achieve:
\begin{defn}
    If $n=\prod_{i=1}^k p_i^{e_i}$, define the \textit{radical} of $n$ as $\rad n=\prod_{i=1}^k p_i$.
\end{defn}
A useful notation that will assist our analysis is the following:
\begin{defn}
    Let $n=\prod_{i=1}^k p_i^{e_i}$, and define the multiset $S=\{x_1, x_2, \cdots, x_k\}$. We define $H_S(n):=\max\left(\frac{e_i}{x_i}\right)$, and set $H(n):=H_{\{1,1,\ldots, 1\}}(n)$. 
\end{defn}
Our algorithm hinges on the following theorem, due to Euler:
\begin{thm}{(Euler)}
    Let $a,n$ be relatively prime positive integers. Then $a^{\varphi(n)}\equiv 1\pmod{n}$.
\end{thm}
Another result due to Euler which we need to use in our analysis is the Euler Product formula:
\begin{thm}{(Euler)}
    Let $a(n): \mathbb{N}\to\mathbb{C}$ be a multiplicative function, i.e., $a(mn)=a(m)a(n)$ is true when $\gcd(m,n)=1$. Then $$\sum_{n\ge 1}\frac{a(n)}{n^s}=\prod_{p~\mathrm{prime}}\sum_{k=0}^{\infty}\frac{a(p^k)}{p^{ks}}.$$
\end{thm}
In particular, we have the following Euler Product for $\zeta$:
$$\zeta(s)=\prod_{p~\text{prime}}\frac{1}{1-p^{-s}}.$$
To acquire precise asymptotics for specific sums, we need the following asymptotics:
\begin{thm}(Stirling's Approximation)
    $\log(x!)=x\log x-x+O(\log x)$
\end{thm}
\begin{thm}(Chebyshev)
    $\vartheta(x)=O(x)$
\end{thm}
To approximate average order summations, we will need Abel's summation formula, referred to as partial summation:
\begin{thm}(Abel)
    Let $(a_n)^{\infty}_{n=0}$ be a sequence of complex numbers. Define $A(t)=\sum_{0\le n\le t}a_n$. Fix real numbers $x<y$ and let $\phi$ be a continuously differentiable function on $[x,y]$. Then $$\sum_{x<n\le y}a_n\phi(n)=A(y)\phi(y)-A(x)\phi(x)-\int_x^y A(u)\phi'(u)\mathrm{d}u.$$
\end{thm}
We compare our algorithm to the standard repeated squaring method for modular exponentiation (there is no general asymptotic improvement to this method, as far as we know). We may compute $a^n\pmod{m}$ in $O(\log n)$ steps by this method. By Euler's theorem, we can reduce this to $O(\log(\varphi(m)))$ which is considered $O(\log m)$.
\newline \newline
For a ring $R$, we let $GL_n(R)$ denote the general linear group over $R$. We will need the well-known fact that $|GL_n(\mathbb{F}_p)|=\prod_{i=0}^{n-1} (p^n-p^i)$. We will also need the following theorem due to Lagrange:
\begin{thm}\label{lagrange} (Lagrange)
    Let $G$ be a group of order $n$. Then for every $a\in G$, $a^n$ is the identity.
\end{thm}
To compute the order of general linear groups over $\mathbb{Z}/T\mathbb{Z}$, we must first compute the order of $GL_n(\mathbb{Z}/p^k\mathbb{Z})$. We do this with the following lemma:
\begin{lemma}\label{glpk}
    For a prime $p$ and a $k\ge 0$ we have $$|GL_n(\mathbb{Z}/p^k\mathbb{Z})|=p^{(k-1)n^2}|GL_n(\mathbb{F}_p)|=p^{(k-1)n^2}\prod_{i=0}^{n-1}(p^n-p^i).$$
\end{lemma}
\begin{proof}
    We proceed by induction on $k$, with $k=1$ trivial. Notice the natural ring homomorphism $\varphi:\mathbb{Z}/p^k\mathbb{Z}\to\mathbb{Z}/p^{k-1}\mathbb{Z}$ given by $a+p^k\mathbb{Z}\mapsto a+p^{k-1}\mathbb{Z}$. This induces the surjection $\varphi_n:GL_n(\mathbb{Z}/p^k\mathbb{Z})\to GL_n(\mathbb{Z}/p^{k-1}\mathbb{Z})$. Hence $|GL_n(\mathbb{Z}/p^k\mathbb{Z})|=|\ker(\varphi_n)||GL_n(\mathbb{Z}/p^{k-1}\mathbb{Z})|$. Since there are $p$ choices for each entry of a matrix in the kernel of $\varphi_n$, the size of the kernel is $p^{n^2}$, and we have the desired result.
\end{proof}
This allows us to prove the following important lemma:
\begin{lemma}\label{glt}
    If $T=p_1^{t_1}p_2^{t_2}\cdots p_k^{t_k}$ is the canonical factorization of some $T\in \mathbb{N}$, 
    $$|GL_n(\mathbb{Z}/T\mathbb{Z})|=\prod_{j=1}^k p_j^{(t_j-1)n^2}\prod_{i=0}^{n-1}(p_j^n-p_j^i).$$
\end{lemma}
\begin{proof}
    By the Chinese Remainder Theorem, we have the ring isomorphism
    $$\mathbb{Z}/T\mathbb{Z}\to \prod_{j=1}^k \mathbb{Z}/p_j^{t_j}\mathbb{Z},$$
    and thus there is a corresponding isomorphism
    $$GL_n(\mathbb{Z}/T\mathbb{Z})\to \prod_{j=1}^k GL_n(\mathbb{Z}/p_j^{t_j}\mathbb{Z}).$$
    This implies the desired result by Lemma \ref{glpk}.
\end{proof}
\begin{rem}
    It follows from Lemma \ref{glt} that we may compute $|GL_n(\mathbb{Z}/T\mathbb{Z})|$ in $O(kn)$ steps, which is insignificant.
\end{rem}
We use standard convention (from \cite{knuth97}) for the asymptotic notations $O(\bullet), o(\bullet), \ll, \gg, \sim$, and $\Omega(\bullet)$. We use $<O(\bullet)$ and $>O(\bullet)$ rather unconventionally: $f(x)<O(g(x))$ provided there is a function $h(x)=O(g(x))$ such that $f(x)<h(x)$ for sufficiently large $x$. We use a similar definition for $>O(\bullet)$. 
\section{The Algorithm}
Our main theorem is the following:
\begin{thm} \label{main}
     Let $m$ be a positive integer with known factorization $p_1^{e_1}p_2^{e_2}\cdots p_k^{e_k}$ and let $a,n$ be positive integers such that $\gcd(a,m)=1$. For each $1\le i\le k$, choose an integer parameter $t_i\in [1,e_i]$. With $\mathcal{T}=\{t_1, t_2, \cdots, t_k\}$, we may compute $a^n\pmod{m}$ in
        $$O\left(H_{\mathcal{T}}(m)+\sum_{i=1}^k t_i\log p_i\right)$$
        steps.
\end{thm}
First, we need a definition that will simplify our calculation of modular inverses.
\begin{defn}\label{inversepair} For some $a$ and $m= \prod p_i^{e_i}$, define $\{u_a, v_a\}$ to be the \textit{inverse pair} of $a$ modulo $m$, where $v_a = \prod p_i^{\nu(a, p_i)}$ and $u_a \equiv (a/v_a)^{-1} \pmod{m}$. Define the inverse pair of $0$ as $\{0,0\}$.
\end{defn}
\begin{ex}
    Let us compute the inverse pair of $12$ modulo $20$ as an example. We have $v_{12} = 2^{\nu(12, 2)} \cdot 5^{\nu (12, 5)} = 4$. Then, $u_{12} \equiv (12/4)^{-1}\equiv 7 \pmod{20}$. So the inverse pair of $12$ modulo $20$ is $\{4, 7\}$. Observe also that $4 \cdot 7^{-1} \equiv 12 \pmod{20}$.
\end{ex}
Notice that $u_a$ always exists because $a/v_a$ is clearly invertible modulo $m$. Intuitively, inverse pairs are inverses equipped with an extra parameter that allows for computation of fractions with denominator not necessarily relatively prime to $m$.
\newline \newline
Consider the following lemma that allows for recursive computation of modular inverses.
\begin{lemma} \label{recursioninverse} If all of $1,2,\dots,a$ are invertible modulo $m$, $$a^{-1}\equiv -(m\pmod{a})^{-1}\Big\lfloor\frac{m}{a}\Big\rfloor\pmod{m}.$$ \end{lemma}
\begin{proof}
Note $m=a\lfloor \frac{m}{a}\rfloor+(m\pmod a)$. Working in $\mathbb{Z}/m\mathbb{Z}$, $$(m\pmod a)^{-1}\cdot \left(m-\Big\lfloor \frac{m}{a}\Big\rfloor\right)=\frac{m-\lfloor\frac{m}{a}\rfloor}{m-a\lfloor \frac{m}{a}\rfloor}=\frac{-\lfloor \frac{m}{a}\rfloor}{-a\lfloor \frac{m}{a}\rfloor}=\frac{1}{a},$$ as desired. \end{proof}
We now extend this result from a recursive computation of modular inverses to a recursive computation of inverse pairs.
\begin{lemma} \label{recursioninversepairs}
We may linearly compute $\{u_n, v_n\}$ given the inverse pairs $\bigcup_{1\le i < n} \{u_i, v_i\}$.
\end{lemma}
\begin{proof}
We claim that the following formula holds:
$$u_n=\begin{cases}u_{m\% n}\cdot \frac{-\lfloor \frac{m}{n}\rfloor}{v_{m\% n}}\pmod{m} & v_n=1 \\ u_{n/v_n}\pmod{m} & v_n>1.\end{cases}$$
Note that $\gcd(n/v_n, m)=1$ by definition, so that
$$u_{n/v_n}\equiv \left(\frac{n}{v_n}\right)^{-1}\equiv u_n\pmod{m}.$$
Suppose $v_n=1$, so that $m$ and $n$ are relatively prime. Decompose $m=qn+r$ via the division algorithm. We wish to show that
$$u_n\equiv u_r\cdot \frac{q}{v_r}\equiv \frac{v_r}{r}\cdot \frac{q}{v_r}\equiv \frac{q}{r}\pmod{m}.$$
Note that
$$u_n\equiv n^{-1},$$
so we wish to show that
$$qn+r\equiv 0\pmod{m},$$
which is obvious. There are no inversion issues as $\gcd(r,m)=1$ by the Euclidean Algorithm, and $\gcd(v_r, m)=1$ as $v_r\mid r$. 
\end{proof}
We are now equipped to prove Theorem \ref{main}. The algorithm is as follows.
\begin{proof}
First, define $T = p_1^{t_1}p_2^{t_2} \cdots p_i^{t_i}$ and $\Phi = \varphi(T)$. Decompose $n=M\Phi+r$ with the division algorithm. Then,
$$a^n\equiv a^{M\Phi+r}\equiv  a^r\cdot (a^{\Phi})^M \pmod{m}.$$
Now, $a^r$ can be computed with the standard Repeated Squaring Algorithm. The complexity of this is $O(\log r)$. By the division algorithm, $r<\Phi<T$, so this is $O(\log T)$. For computation of the second term, first notice that by Euler's theorem, $a^\Phi \equiv 1 \pmod{T}$ so there exists some integer $s$ with $a^{\Phi}=Ts+1$. Now, we expand using the binomial theorem:
$$(a^{\Phi})^M = (Ts+1)^M,$$
$$=1\cdot \binom{M}{0}+Ts\cdot \binom{M}{1}+\cdots+(Ts)^i \binom{M}{i}+\cdots+(Ts)^M\binom{M}{M}.$$
Let $\ell = 1 + \text{max}(\lfloor e_i/t_i \rfloor )$. Only the first $\ell$ terms need to be computed, as all the later terms are $0$ modulo $m$. This is because for a prime $p_i$, $$\nu_{p_i}(T^\ell) = \nu_{p_i}(\left(p_i^{t_i}\right)^\ell )= \nu_{p_i}(p_i^{t_i ( \lfloor e_i/t_i \rfloor+1)}  ) = t_i ( \lfloor e_i/t_i \rfloor +1)\ge t_i(e_i/t_i) = e_i,$$ so $p_i^{e_i}\mid  T^\ell$ for all $i$, and thus $m \mid T^\ell$ and $T^\ell \equiv 0 \pmod{m}$.
\newline \newline
Additionally, we can linearly compute the $\binom{M}{i}$ terms with the identity
$$\binom{M}{i+1}=\frac{M!}{(i+1)!(M-i-1)!},$$
$$=\frac{M-i}{i+1}\frac{M!}{i!(M-i)!}=\frac{M-i}{i+1}\binom{M}{i}.$$
Notice that in our binomial recursion, the one inverse we need to calculate $\binom{M}{i}$ is the inverse of $i$. Using Lemma \ref{recursioninversepairs}, because we are performing a linear computation, we can compute the inverse pairs of $0$ through $\ell-1$ modulo $m$ in $O(\ell)$ time and $O(\ell)$ space. This is the only time we use more than $O(1)$ memory in this algorithm. From here, we can use our identity to compute $\binom{M}{0}$ to $\binom{M}{\ell - 1}$ in $O(\ell)$ time and space. This is the final aspect of our calculation, and we may now recover $a^n\pmod{m}$. It is now clear that the number of steps is $O(\ell+\log T)$ and the space complexity is $O(\ell)$. Because $\ell=1+H_{\mathcal{T}}(m)$ and
$$\log T = \log\left(\prod_{i=1}^k p_i^{t_i}\right)=\sum_{i=1}^k\log(p_i^{t_i})=\sum_{i=1}^k t_i\log p_i,$$
this concludes the proof of our main theorem.
\end{proof}
\begin{ex}
    We present the details of the algorithm in the special case of $k=1$, $p_1 = p$, $e_1=e$, and $t_1 = 1$ (so $T = p$). Our goal is to compute $a^n \pmod{p^e}$. We let $n = (p-1)m + r$, so that \[a^n \equiv (a^{p-1})^m a^r.\]
    Notice that $a^{p-1} \equiv 1 \pmod{p}$, so we let $a^{p-1} = 1 + sp$. Compute $a^r$ via the standard $\textsc{modExp}$ algorithm. Now compute
    $$(a^{p-1})^m=(1+sp)^m\equiv 1+\binom{m}{1}(sp)^1+\binom{m}{2}(sp)^2+\cdots+\binom{m}{e-1}(sp)^{e-1}\pmod{p^e},$$
    by the Binomial Theorem. We can recursively compute the $\binom{m}{i}$ coefficients modulo $p^e$ by the method outlined above. Thus we may compute each $\binom{m}{i}(sp)^i$ term for $0\le i\le e-1$. Finally, compute $a^n\equiv (a^{p-1})^m a^r\pmod{p^e}$.
\end{ex}
\begin{ex}
    As a particular case, let us compute $7^{123}\pmod{11^3}$. Decompose $123=12\cdot 10+3$. Compute $7^{3}\equiv 343\pmod{11^3}$. Write $s=\frac{7^{10}-1}{11}$. Notice $s$ is an integer by Euler. Then
    $$7^{120} \equiv (7^{10})^{12}\equiv(1+11s)^{12}=\binom{12}{0}+\binom{12}{1}\cdot 11s+\binom{12}{2} \cdot (11s)^2\equiv 23 \pmod{11^3}.$$
    Therefore, $7^{123}\equiv 343 \cdot 23 \equiv 1234\pmod{11^3}$.
\end{ex}
Notice that this algorithm is $O(H_{\mathcal{T}}(m))$ space, which is not always ideal. For more memory-sensitive scenarios, we may compute inverses directly rather than recursively. This gives rise to the following theorem:
\begin{thm}\label{main-memory-optimized}
    Let $m$ be a positive integer with known factorization $p_1^{e_1}p_2^{e_2}\cdots p_k^{e_k}$ and let $a,n$ be positive integers such that $\gcd(a,m)=1$. For each $1\le i\le k$, choose an integer parameter $t_i\in [1,e_i]$. With $\mathcal{T}=\{t_1, t_2, \cdots, t_k\}$, we may compute $a^n\pmod{m}$ in
        $$O\left(H_{\mathcal{T}}(m)\log(H_{\mathcal{T}}(m))+\sum_{i=1}^k t_i\log p_i\right)$$
        steps and $O(1)$ memory.
\end{thm}
\begin{proof}
    Identical to Theorem \ref{main}, except instead of recursively computing inverse pairs, compute the modular inverses $u_n$ with standard inversion algorithms: $u_n \equiv (n/v_n)^{-1} \pmod{m}$. If, for example, we use the Extended Euclidean Algorithm then calculating the inverses of $1$ to $\ell$ has complexity $O(\sum_{n=1}^{\ell} \log (n/v_n))=O(\sum_{n=1}^{\ell} \log n) =O( \ell \log \ell)$ We do not give pseudocode for this algorithm, instead see \cite{isaacs24}.
\end{proof}
\subsection{Pseudocode}
Define $\textsc{modExp}(a, b, c)$ to be the standard Repeated Squaring Algorithm for computing $a^b \pmod{c}$. We use the integer division notation $a//b=\lfloor \frac{a}{b}\rfloor$. The pseudocode builds up to our algorithm, $\textsc{FastModExp}(a$, $n$, $P$, $E$, $T)$. Here $a$ and $n$ are positive integers, as usual. $P$ and $E$ are equal sized arrays, $P$ is an array of primes and $E$ is an array of positive integers. $T$ also has the same length as $P$ and $E$, and is an array of non-negative integer parameters such that $T[i]\le E[i]$ for all $i$.  We store the inverse pairs in a 2-dimensional array $L$. This algorithm computes $a^n$ modulo $m = \prod_{i=1}^{\mathrm{len}(P)} P[i]^{E[i]}$.
\begin{algorithm}
  \caption{Computes $\nu (n, p)$}
  \label{vp}
  \begin{algorithmic}[H]
    \Function{$\nu$}{$n$, $p$}
      \If{$n\%p = 0$}
        \State \Return $1 + \nu(n//p, p)$
      \EndIf
      \State \Return $0$
    \EndFunction
  \end{algorithmic}
\end{algorithm}

\begin{algorithm}
  \caption{Computes the inverse pair of $i$ given inverse pairs of $0$ through $i-1$}
  \label{nextinversepair}
  \begin{algorithmic}[H]
    \Function{nextInversePair}{$i$, $m$, $L$, $P$}
      \State $v \gets 1$
      \State $j \gets 0$
      \While {$j< \text{len}(P)$}
      \State $v \gets v \cdot P[j]^{\nu(i, P [ j ] )}$
      \State $j \gets j+1$
      \EndWhile
      \If{$v = 1$}
        \State $u \gets (L[m \% i][0] \cdot  ((m - m // i) // L[m \% i][1]) )\% m$
      \EndIf
      \If{$v \ne 1$}
        \State $u \gets L[i//v][0]$
      \EndIf
      \State \Return $[u,v]$
      
    \EndFunction
    
  \end{algorithmic}
\end{algorithm}

\begin{algorithm}
  \caption{Computes the inverse pairs of $0$ through $i$}
  \label{nextinversepair}
  \begin{algorithmic}[H]
    \Function{generateInversePairs}{$i$, $m$, $P$}
    \State $L \gets [[0,0],[1,1]]$
    \State $j \gets 2$
    \While{$j < i + 1$}
        \State $L \gets L + \textsc{nextInversePair}(j\text{, }m\text{, }L\text{, }P)$
        \State $j \gets j + 1$
    \EndWhile
    \State \Return $L$
    \EndFunction
    
  \end{algorithmic}
\end{algorithm}
\begin{algorithm}
  \caption{Computes $a^n$ modulo $m = \prod P[i]^{E[i]}$}
  \label{nextinversepair}
  \begin{algorithmic}[H]
    \Function{FastModExp}{$a$, $n$, $P$, $E$, $T$}
    \State $t \gets 1$
    \State $\phi \gets 1$
    \State $m \gets 1$
    \State $i \gets 0$
    \While{$i < \text{len}(P)$}
        \State $\text{temp} \gets P[i]^{T[i] - 1}$
        \State $\phi =\phi \cdot \text{temp} \cdot (P[i] - 1)$
        \State $t = t \cdot \text{temp} \cdot P[i]$
        \State $m = m \cdot P[i]^{E[i]}$
        \State $i\gets i+1$
    \EndWhile
    \State $r \gets n \% \phi$
    \State $q \gets (n - r) // \phi$
    \State $c \gets \textsc{modExp}(a, \phi, m) - 1$
    \State $\text{sum} \gets 0$
    \State $\text{choose} = 1$
    \State $\text{cExp} \gets 1$
    \State $\ell \gets 0$
    \State $i \gets 0$
    \While {$i < \text{len}(P)$}
        \If{$\ell < E[i]//T[i]$}
            \State {$\ell \gets E[i]//T[i]$}
        \EndIf
        \State $i\gets i+1$
    \EndWhile
    \State $\text{inverses} \gets \textsc{generateInversePairs}(\ell\text{, }m\text{, }P)$
    \State $i \gets 0$
    \While{$i < \textsc{min}(\ell\text{, }q + 1)$}
        \State $\text{sum} \gets (\text{sum} + (\text{choose} \cdot \text{cExp})) \% m$
        \State $\text{cExp} \gets (\text{cExp} \cdot c ) \% m$
        \State $\text{choose} \gets (((\text{choose} \cdot (q - i)) \% m ) // \text{inverses}[i+1][1] \cdot \text{inverses}[i+1][0]) \% m$
        \State $i\gets i+1$
    \EndWhile
    \State $\text{ar} \gets \textsc{modExp}(a, r, m)$
    \State \Return $(\text{sum} \cdot \text{ar}) \% m$
    \EndFunction
    
  \end{algorithmic}
\end{algorithm}
\newpage \newpage \newpage \newpage \newpage
\section{Mathematical Analysis}
It's important to optimally choose the parameters $t_i$. For general $m$, the optimal choice is of the form $t_i=O(1)$ for all $i$. This is due to the following theorem:
\begin{thm}{\cite{Niven}}\label{niv} We have that
    $$\lim_{n\to\infty}\frac{1}{n}\sum_{k\le n}H(k)=C,$$
    where $C=1+\sum_{k\ge 2}\left(1-\frac{1}{\zeta(k)}\right)\approx 1.705$ is Niven's constant.
\end{thm}
In other words, for general $m$ we expect $H(m)$ to be $O(1)$. The choice of $t_i=O(1)\forall i$ is then clear. Let us proceed with some analysis of this choice.
\newline \newline
By some metrics, we make asymptotic improvement to the standard $O(\log m)$ repeated squaring methods with $t_i=1$ (i.e., $T=\rad m$). But by other metrics, we may not. First, let us sweep the $H(m)$ term under the rug (we can do so due to Theorem \ref{niv}). If we just compute $\sum_{m\le x}\frac{\log m}{\log \rad m}$ or $\sum_{m\le x}(\log m - \log \rad m)$, we get $\Omega(x)$. However, the sum $\sum_{m\le x}\frac{m}{\rad m}$ is not $O(x)$, and is not even $O(x(\log x)^A)$ for any $A$.
\begin{thm}\label{badone}
    For any $x\in \mathbb{N}$, we have that
    $$\sum_{m\le x}\frac{\log m}{\log \rad m}=\Omega(x)=\sum_{m\le x}(\log m-\log \rad m)$$
\end{thm}
\begin{proof}
    By the simple estimate that $\frac{a}{b}\le a-b+1$ for $a\ge b\ge 1$, we have that
    $$\sum_{m\le x}\frac{\log m}{\log \rad m}\le O(x)+\sum_{m\le x}(\log m-\log\rad m).$$
    Note that, by Stirling's approximation,
    $$\sum_{m\le x}\log m=\log(x!)=x\log x+O(x),$$
    and
    $$\sum_{m\le x}\log\rad m=\sum_{m\le x}\log\left(\prod_{p\mid m}p\right)=\sum_{m\le x}\sum_{p\mid m}\log p.$$
    Swapping the order of the summation, we may rewrite this sum as
    $$\sum_{p\le x}\log p\left\lfloor \frac{x}{p}\right\rfloor=\sum_{p\le x}\frac{x}{p}\log p-O\left(\sum_{p\le x}\log p\right)=x\sum_{p\le x}\frac{\log p}{p}-O(\vartheta(x)).$$
    Thus, applying Merten's first theorem and Chebyshev's asymptotics for $\vartheta$, we have that
    $$\sum_{m\le x}\log \rad m=x\log x-O(x).$$
    Therefore,
    $$\sum_{m\le x}\frac{\log m}{\log \rad m}=O(x).$$
    On the other hand, $\log m\ge\log \rad m$, so that
    $$\sum_{m\le x}\frac{\log m}{\log \rad m}\ge x.$$
    Therefore,
    $$\sum_{m\le x}\frac{\log m}{\log \rad m}=\Omega(x),$$
    as desired.
\end{proof}
\begin{thm}\label{goodone}
    For any $x\in \mathbb{N}$ and $A>0$, we have that
    $$\sum_{m\le x}\frac{m}{\rad m}\ne O(x(\log x)^A)$$
\end{thm}
\begin{proof}
    Notice that $f(m):=\frac{m}{\mathrm{rad}(m)}$ is multiplicative, and $f(p^k)=p^{k-1}$ for prime $p$ and $k\in \mathbb{N}$. Therefore, the Dirichlet Series of $f$ is
$$F(s)=\sum_{n=1}^{\infty}\frac{f(n)}{n^s}=\prod_{p}\left(1+\sum_{k\ge 1}\frac{p^{k-1}}{p^{ks}}\right).$$
This expression simplifies to
$$F(s)=\prod_{p}\left(1+\frac{1}{p^s-p}\right).$$
Thus
$$\frac{F(s)}{(\zeta(s))^{A+1}}=\prod_{p}\left(\left(1+\frac{1}{p^s-p}\right)(1-p^{-s})^{A+1}\right)$$
by the Euler product for $\zeta$. Sending $s\to 1^+$ on the real axis, the product tends to infinity. Because $\lim_{s\to 1^+}\zeta(s)^{A+1}(1-s)^{A+1}=1$, $F(s)\ne O\left(\frac{1}{(s-1)^{A+1}}\right)$. Let $S(x)=\sum_{n\le x}f(n)$ and suppose for the sake of contradiction that $S(x)=O(x(\log x)^A)$. For $s>1$ we have that, by partial summation,
$$F(s)=\int_{1^-}^{\infty}\frac{\mathrm{d}S(x)}{x^s}=\left[\frac{S(x)}{x^s}\right]_{1^-}^{\infty}+\int_1^{\infty}\frac{S(x)}{x^{s+1}}\mathrm{d}x.$$
Note that because $S(x)=O(x(\log x)^A)$ the first term on the RHS vanishes. Thus $F(s)$ is on the order of
$$\int_1^{\infty}\frac{(\log x)^A}{x^s}\mathrm{d}x.$$
To arrive at the desired contradiction, we must show that this integral is $O\left(\frac{1}{(s-1)^{A+1}}\right)$. In order to do this, we induct on $A$ (we may assume $A$ is an integer). For $A=0$, the result is obvious. For higher $A$, we may integrate by parts. Set $u=(\log x)^A$ and $\mathrm{d}v=x^{-s}\mathrm{d}x$. Then
$$\int \frac{(\log x)^A}{x^s}\mathrm{d}x=\frac{(\log x)^Ax^{1-s}}{1-s}-\frac{A}{1-s}\int \frac{(\log x)^{A-1}}{x^s}\mathrm{d}x,$$
which implies the result by the induction hypothesis, as both terms are $O(x(\log x)^A)$.
\end{proof}
In fact, there is a more precise estimate than Theorem \ref{goodone}. 
\begin{thm}{\cite{ROBERT2013802}}\label{advancedthm}
    For any $x\in\mathbb{N}$ we have that
    $$\sum_{m\le x}\frac{m}{\rad m}=x\exp\left((1+o(1))\sqrt{\frac{8\log x}{\log\log x}}\right)$$
\end{thm}
By the metric given by Theorem \ref{badone}, we make an $O(1)$ improvement ``on average" to the repeated squaring method. However, by Theorem \ref{goodone}, we ``expect" $m \gg (\log m)^A (\rad m)$ (very loosely), so that we ``expect"
 $\log m \gg \log \rad m + A\log\log m$. By this metric, we do make asymptotic improvement to the repeated squaring method. 
 \newline \newline
 Nonetheless, it is particularly fruitful to work with smooth $m$ rather than general $m$. We have the following corollary of Theorem \ref{main}:
\begin{cor}\label{sqrt}
    Let $m$ be a positive integer with known factorization $p_1^{e_1}p_2^{e_2}\cdots p_k^{e_k}$ such that all primes $p_i$ have the same bit length as an integer $P$, and all $e_i$ are such that $e_i\sim k\log P$. Let $a,n$ be positive integers such that $\gcd(a,m)=1$. We may then compute $a^n\pmod{m}$ in $O(\sqrt{\log m})$ steps.
\end{cor}
\begin{proof}
    The idea is to use Theorem \ref{main} and set $H_{\mathcal{T}}(n)\approx\mathsf{c}$ for some $\mathsf{c}$, and choose $t_i$ such that $\frac{e_i}{t_i}\approx \mathsf{c}$ for all $i$. In other words, take $t_i=\lfloor \frac{e_i}{\mathsf{c}}\rfloor$, for example. Then $H_{\mathcal{T}}(n)=\mathsf{c}+O(1)$, and
    $$\sum_{i=1}^k t_i\log p_i=\frac{1}{\mathsf{c}}\sum_{i=1}^k(e_i\log p_i)+O\left(\sum_{i=1}^k \log p_i\right)=\frac{\log m}{\mathsf{c}}+O(\log \rad m).$$
    Notice that
    $$\sqrt{\log m}=\sqrt{\log P\sum_{i=1}^k e_i}\sim k\log P,$$
    so we may choose $\mathsf{c}$ on the order of $\sqrt{\log m}$ (as $\mathsf{c}<e_i$ is necessary and sufficient). With this choice, we compute $a^n\pmod{m}$ in
    $$O(\sqrt{\log m})+O(\log \rad m)$$
    steps. Note that
    $$\log \rad m\sim k\log P\sim \sqrt{\log m}$$
    so that our modular exponentiation can be completed in $O(\sqrt{\log m})$ steps.
\end{proof}
\begin{rem}
    This shows that in general it is best to choose all $t_i$ as $O(1)$. When faced with the problem of optimizing these parameters, we can hence choose a relatively strong upper bound. On the other hand, it seems very difficult to compute the exactly optimal multiset $\mathcal{T}$. Note that this is important because a smart choice of $t_i$ can make for a big performance boost (see Figure \ref{fig:optimalt}). If one could better understand the constant factors at play, then, modifying the above calculations, they could theoretically find the optimal choice of $\mathcal{T}$.
\end{rem}
In other words, when $m$ has large exponents relative to its prime factors, our algorithm makes large improvements (as $\log \rad m\ll \log m$ in the case where the exponents are on the order of $\log P$). In particular, our algorithm does well for prime powers. We have the following cororollary:
\begin{coror}
    Let $m=p^\ell$ where $\ell=O(\log p)$. Let $a,n$ be positive integers such that $\gcd(a,m)=1$. We may then compute $a^n\pmod{m}$ in $O(\log p)$ steps.
\end{coror}
\begin{proof}
    This trivially follows from Corollary \ref{sqrt} with $k=1$.  
\end{proof}
This is quite impressive as, if we pick $\ell$ on the order of $\log p$, we may modular exponentiate modulo $p^{\ell}$ in the same number of steps (asymptotically) as modulo $p$. Furthermore, because each operation (modular multiplication or addition) can be taken to be $O((\log m)^2)$, we have the following result:
\begin{cor}
    Let the modulus $m$ be a positive integer with known factorization $p_1^{e_1}p_2^{e_2}\cdots p_k^{e_k}$ such that all primes $p_i$ have the same bit length as an integer $P$, and all $e_i$ are such that $e_i\sim k\log P$. If the standard modular exponentiation algorithm takes $T$ time, our algorithm for the same values takes $c T^{5/6}$ time, for some constant $c$ and given unit of time.
\end{cor}
\begin{proof}
For convenience, denote $O((\log m)^e)$ simply as $f(e)$. The standard algorithm takes $f(1)$ steps and hence $f(3)$ time. Our algorithm takes $f(1/2)$ steps and hence $f(5/2)$ time. Then, if $T_1$ is the time it takes our algorithm to run and $T_2$ is the time for the standard algorithm to run, we have for some constants $c_1, c_2$ that
\[
(\log m)^{5/2} \sim c_1T_1,
\]
\[
(\log m)^{3} \sim c_2T_2.
\]
Hence $T_1 \sim c_2^{5/6} c_1^{-1}T_2^{5/6}$, so we obtain the desired result by taking the appropriate $c \sim c_2^{5/6} c_1^{-1}$.
\end{proof}
\begin{rem}
    One may think that for such smooth $m$, a CRT-type approach may also be fast. However, it is well-known (see e.g. \cite{inproceedings}) that CRT is $O(\ell^2)$, where the moduli have bit length $\ell$. Hence this is a non-issue.
\end{rem}
\section{General Modular Exponentiation}
We may extend the ideas of our method to prove a more general theorem about matrix modular exponentiation, which translates over to linear recurrences. In this section, a step is a matrix multiplication.
\begin{thm}\label{matmodexp}
    Let $m=p_1^{e_1}p_2^{e_2}\cdots p_{k}^{e_k}$ be the canonical factorization of some $m\in \mathbb{N}$. Choose parameters $t_i\in [1, e_i]$ for each $1\le i\le k$. Let $d\in \mathbb{N}$. Let $A\in GL_d(\mathbb{Z}/m\mathbb{Z})$ and $n\in \mathbb{N}$. Then we may compute $A^n$ in $O\left(\max(e_i/t_i)+\sum_{j=1}^k (t_j-1)d^2\log p_j + \sum_{j=1}^k \sum_{i=0}^{d-1} \log(p_j^d-p_j^i)\right)$ steps.
\end{thm}
\begin{proof}
    Let $T=\prod_{j=1}^{k} p_j^{t_j}$. The algorithm is identical to that of Theorem \ref{main}, except we decompose the exponent modulo $|GL_d(\mathbb{Z}/T\mathbb{Z})|$ instead of $\varphi(T)$. This works because of Lagrange's theorem \ref{lagrange}. The complexity is hence $O\left(\max(e_i/t_i)+\log |GL_d(\mathbb{Z}/T\mathbb{Z})|\right)$. By Lemma \ref{glt}, this is the desired complexity.
\end{proof}
\begin{rem}
    A near-identical result holds for the special linear group instead. This is less practically useful but still important to note.
\end{rem}
This asymptotic may be slightly difficult to work with, but a trivial corollary is the following:
\begin{cor}\label{nicematmodexp}
    Let $m=p_1^{e_1}p_2^{e_2}\cdots p_{k}^{e_k}$ be the canonical factorization of some $m\in \mathbb{N}$. Choose parameters $t_i\in [1, e_i]$ for each $1\le i\le k$. Let $d\in \mathbb{N}$. Let $A\in GL_d(\mathbb{Z}/m\mathbb{Z})$ and $n\in \mathbb{N}$. Then we may compute $A^n$ in $O\left(\max(e_i/t_i)+d^2\sum_{i=1}^k t_i\log p_i\right)$ steps.
\end{cor}
\begin{proof}
    This follows from Theorem \ref{matmodexp} and the obvious estimate $\log(p_j^d-p_j^i)<\log(p_j^d)=d\log p_j$. Indeed, the double sum is upper-bounded by $d^2\sum_{j=1}^k \log p_j$ and so the result follows.
\end{proof}
\begin{rem}
    With $d=1$, we get Theorem \ref{main} for prime powers. Also, this shows we can modular exponentiate matrices of size $O(1)$ in the same complexity as natural numbers.
\end{rem}
Another corollary is that we may quickly compute the residue of large elements of a linear recurrent sequence modulo prime powers:
\begin{cor}
    Let $m=p_1^{e_1}p_2^{e_2}\cdots p_k^{e_k}$ be the canonical factorization of some $m\in \mathbb{N}$. Choose parameters $t_i\in [1, e_i]$ for each $1\le i\le k$. Let $(u_n)^{\infty}_{n=0}$ be a sequence of elements of $\mathbb{Z}/m\mathbb{Z}$ related with a linear recurrence relation of degree $d$:
    $$u_{n+d}=c_{d-1}u_{n+d-1}+\cdots+c_0u_n.$$
    Suppose that $\gcd(c_0, m)=1$. Given sufficient initial terms, we may compute any element $u_N$ in $O\left(\max(e_i/t_i)+\sum_{j=1}^k (t_j-1)d^2\log p_j + \sum_{j=1}^k \sum_{i=0}^{d-1} \log(p_j^d-p_j^i)\right)$ steps.
\end{cor}
\begin{proof}
    Linear recurrence relations of order $d$ can be represented as matrix powers. In particular, the respective matrices are in $GL_{d}(\mathbb{Z}/m\mathbb{Z})$, so the result follows immediately by Theorem \ref{matmodexp}. One must make sure that the matrix produced has determinant invertible in $\mathbb{Z}/m\mathbb{Z}$ (and hence in $\mathbb{Z}/T\mathbb{Z}$). This determinant has magnitude $c_0$, so we require $\gcd(c_0, m)=1$. This condition is met, so we are done.
\end{proof}
\begin{rem}
    This shows that we may compute the $n$th Fibonacci number modulo $m$ in the same number of steps as we perform a modular exponentiation modulo $m$. For example, modulo $p^k$ we may do it in $O(k+\log p)$ steps.
\end{rem}
Fiduccia \cite{doi:10.1137/0214007} and, more recently, Bostan and Mori \cite{bostan2020simple} provide the state-of-the-art results for this problem in the case of a sequences over a general ring. The amount of steps taken is on the order of $M(d)\log n$, where $M(d)=O(d\log d\log\log d)$ is the number of operations to multiply two polynomials in the ring. We achieve a stronger bound for the case where the ring is $\mathbb{Z}/m\mathbb{Z}$. Our bound does not depend on $n$ because of the reduction we make modulo $|GL_d(\mathbb{Z}/T\mathbb{Z})|$. In order to reduce $n$ in the same manner modulo $m$ for the bounds given by Fiduccia and by Bostan and Mori, the best possible reduction is by $|GL_d(\mathbb{Z}/m\mathbb{Z})|$. By Lemma \ref{glt}, $M(d)\log |GL_d(\mathbb{Z}/m\mathbb{Z})|\approx M(d)d^2\log m\gg O(d^3\log d\log m)$. If $\omega\in [2,3]$ is the exponent for matrix multiplication, we take under $O(d^{\omega}\max(e_i/t_i)+d^{2+\omega}\log T)$ steps of the same complexity as the steps taken by Fiduccia. Therefore Fiduccia's algorithm is better as a function of $d$, but we are better as a function of $m$. It is often the case that $m\gg d$, so we make significant improvement. Indeed, the key optimization that Fiduccia makes (involving the characteristic polynomial) only affects the complexity as a function of $d$.
\newline \newline
We can also apply our algorithm to ring extensions:
\begin{thm}
    Let $p$ be a prime and $k$ and $n$ be natural numbers. Consider a finite ring extension $R=\mathbb{Z}/p^k\mathbb{Z}[\alpha_1, \alpha_2, \cdots, \alpha_n]$. Consider the corresponding field extension $F=\mathbb{F}_p[\alpha_1, \alpha_2, \cdots, \alpha_k]$. We can exponentiate in $R$ in $O(k+\log |F|)$ steps, where each step is an operation in $R$.
\end{thm}
\begin{proof}
    The algorithm is exactly that of Theorem \ref{main} for prime powers and $t=1$, except that we decompose the exponent modulo $|F|$. This works due to Lagrange's Theorem \ref{lagrange}.
\end{proof}
A nice example of this theorem is that we provide fast modular exponentiation for Gaussian Integers modulo prime powers! This case is also related to Theorem \ref{matmodexp} due to the bijection $a+bi\leftrightarrow\begin{pmatrix}a & -b \\ b & a\end{pmatrix}$.
\section{Programmatical Analysis}
In this section, we test only integer modular exponentiation of our algorithm, as we prove that the generalizations have similar complexities.
\newline \newline
Let us begin by testing our algorithm for $m=p^k$. We will test against Python's built-in $\mathrm{pow}$ function. We will iterate over primes $p$, and choose $k$ randomly in a small interval around $\log p$. In particular, we choose $k$ uniformly at random in the interval $[\log p-\sqrt{\log p}, \log p+\sqrt{\log p}]$. We choose $a\in [p^k/2, p^k]$ randomly (such that $\gcd(a,p)=1$) as well. We choose $n$ randomly in the interval $[p^k/2, p^k]$. We choose $t_i=1$ for simplicity. We then compare the runtime for computation of $a^n\pmod{p^k}$ via our method ($\mathcal{R}_1$) and Python's built-in $\mathrm{pow}$ function ($\mathcal{R}_2$) over $1000$ iterations.
\newline \newline
In Figure 1 we iterate over $0\le n\le 3.5\cdot 10^4$, plotting $p_{n+0.7\cdot 10^5}$ versus the respective ratio $\frac{\mathcal{R}_2}{\mathcal{R}_1}$.
\begin{figure}[htbp]
\centering
\includegraphics[scale=0.4]{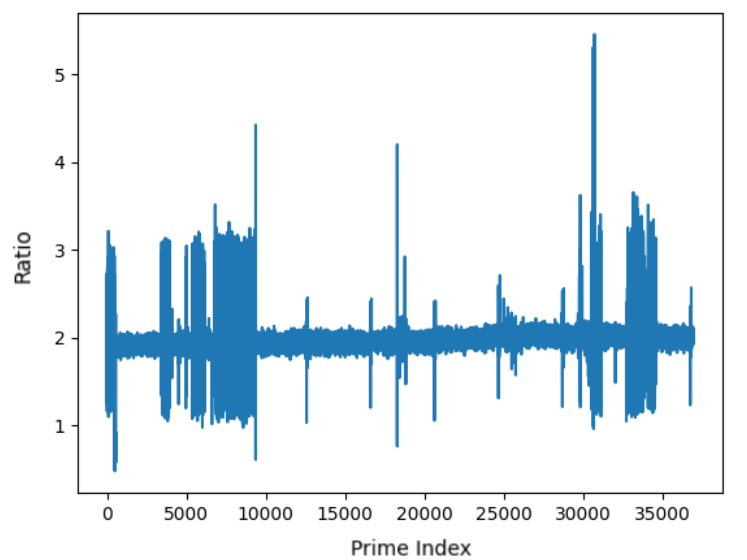}
\caption{Ratio for small primes}
\label{fig:smallprimes}
\end{figure}
In Figure 2 we iterate over $0\le n\le 3.5\cdot 10^4$, plotting $p_{n+3.5\cdot 10^5}$ versus the respective ratio $\frac{\mathcal{R}_2}{\mathcal{R}_1}$.
\begin{figure}[htbp]
\centering
\includegraphics[scale=0.4]{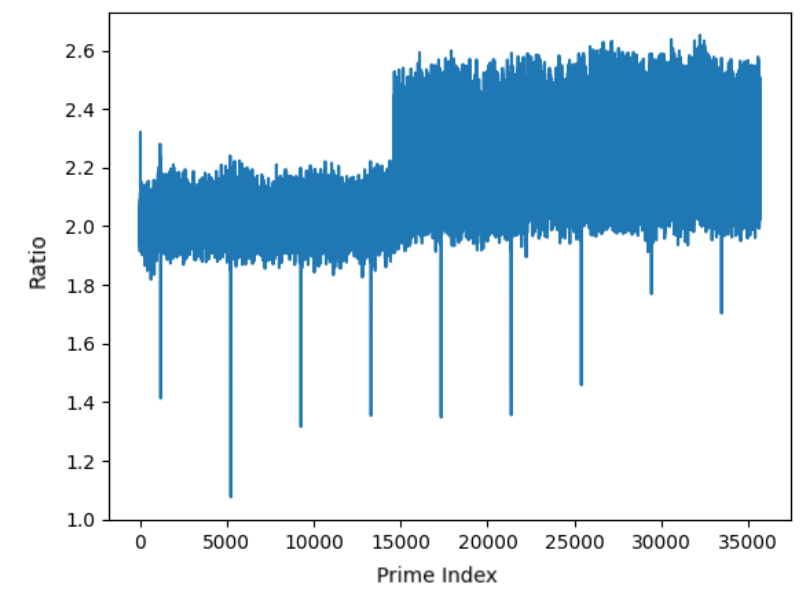}
\caption{Ratio for big primes}
\label{fig:smallprimes}
\end{figure}
As seen in the figures, there is a high variance in the ratio over small primes, whereas it steadies out for larger primes. The large jump in the second figure at about $n=16000$ is likely because Python has different optimizations for smaller calls of the $\mathrm{pow}$ function. We still do not have an explanation for the random jumps in data. Nonetheless, a basic implementation of our algorithm makes significant improvements to the highly optimized $\mathrm{pow}$ function. 

The below figure shows how the problem of optimizing $\mathcal{T}$ is important.
\begin{figure} [H]
    \centering
    \includegraphics[scale=0.4]{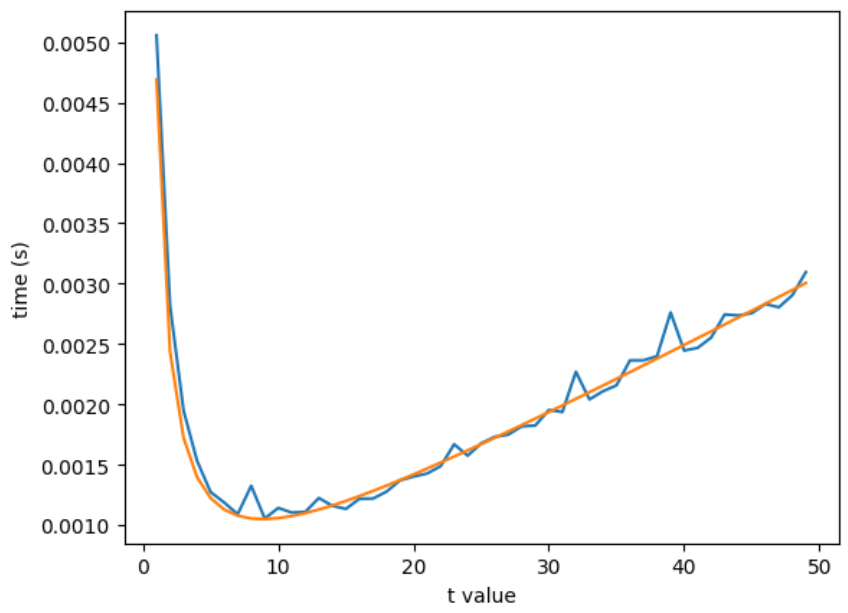}
    
    \caption{Choosing optimal $t$}
    \label{fig:optimalt}
\end{figure}
For a specific choice of computing $a^n\pmod{p^e}$ with $p=101, e=200, a=13,$ and $n=\lfloor \frac{p^e}{3}\rfloor$ (chosen arbitrarily), we vary the choice of $\mathcal{T}$ over the set $\{\{1\}, \{2\}, \cdots, \{50\}\}$. We plot this against the runtime for computation to create the blue curve. As seen in the figure, the optimal value of $\mathcal{T}$ yields a time approximately $5$ times faster than $\mathcal{T}=\{1\}$. The orange curve demonstrates how the graph follows a curve of the form $at + b/t$, as indicated by Corollary \ref{sqrt}. The $R^2$ value is $96.5\%$.
\newline \newline
Recall that we achieve a complexity of $O(\sqrt{\log m})$ for an infinite family of $m$, by Corollary \ref{sqrt}. We aim to show this empirically. Because the complexity of the repeated squaring algorithm is $O(\log m)$, if we graph the ratio $r = \mathcal{R}_2/\mathcal{R}_1$ of the runtime of python's built-in modular exponentiation function to ours, we expect to see $r \propto \sqrt{\log m}$. With $y = r$ and $x = \log m$, we anticipate a graph of the form $y =c \sqrt{x}$.
\newline \newline
The setup is as follows. Let $P(n)$ be the first prime greater than $10^n$. For $10 \le n < 50,$ we let $p = P(n)$ and $e = n$. Pick $t_i=1$ for simplicity. Then, $m = p^e$ is in the desired family. For each $n$, we compute the ratio $r$ and plot it against $\log_{10} m \approx n^2$.
\begin{figure}[H]
\centering
\includegraphics[scale=0.4]{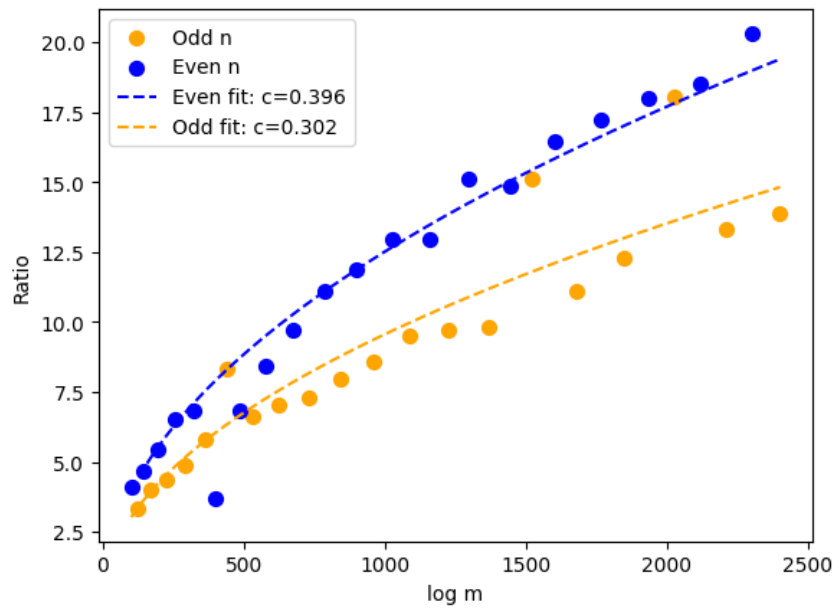}
\caption{Ratio graph with $\log m$ on the $x$-axis}
\label{fig:sqrtx}
\end{figure}
 The desired $y=c\sqrt{x}$ curve appears. Interestingly, the ratio for even $n$ is on average $24.8$ percent higher than for odd $n$. The curves fit quite well: the $R^2$ values for even and odd $n$ are $94.9$ and $85.2$ percent respectively.
\newline \newline
One might notice three values of odd $n$ appear to fit on the curve for even $n$. Those values are $21$, $39$, and $45$. We do not know why this phenomenon occurs or why the ratio is higher in general for even $n$.

\begin{rem}
In Python, we have achieved computation time more than $200$ times faster than the built-in $\mathrm{pow}$ for specific large values of $m$.
\end{rem}
See \cite{isaacs24} for the Python code used to create these graphs.

\section{Conclusion}
We presented a fast algorithm for modular exponentiation when the modulus is known. We also presented a variant of this algorithm which uses less memory. We analyzed this algorithm in the general case, then shifted our focus to the specific case where the modulus has large prime exponents. We showed particular interest in the case where the modulus is a prime power, and we analyzed this case programmatically, testing it against Python's built-in $\mathrm{pow}$ function. We also presented a stronger version of our algorithm for matrix modular exponentiation, which applies to the computation of large terms in linear recurrent sequences modulo some $m$.
\newline \newline
This algorithm has potential practical use in cryptography. Fast modular exponentiation is vital in the fast encryption of classic algorithms such as RSA and the Diffie-Hellman Key Exchange. It is even used in quantum algorithms: modular exponentiation is the bottleneck of Shor's algorithm. If one could construct a cryptosystem in which it is useful to have a known modulus with large prime exponents, our algorithm would be applicable to its encryption process. For example, a variant of Takagi's cryptosystem \cite{takagi1998fast} with larger exponents has such properties. Additionally, work has been done on using matrix exponentiation and linear recurrences for error-correcting codes. For example, Matsui's 2012 paper \cite{matsui2012decoding} uses linear sequences for a decoding algorithm. It is quite possible that our algorithm is potentially useful for such an algorithm.
\newline \newline
There are a couple of things that we wish to do with this work going forward. We'd like to find a framework for programmatically testing general moduli (not only prime powers). Additionally, we hope to to make further progress on the front of optimizing $\mathcal{T}$ in practice. Furthermore, we want to come up with explanations for some of the phenomena that we see in the figures in section 5. Finally, we aim to implement further optimizations to our algorithm such as Montgomery Reduction.
\section{Acknowledgments}
We thank Dr.~Simon Rubinstein-Salzedo for the useful discussion we had during the creation of this paper. We also thank Eva Goedhart and Nandita Sahajpal for awarding us the Lehmer Prize for this work at the 2023 West Coast Number Theory conference. We finally thank Dr.~John Gorman from Jesuit High School Portland for inspiring this paper.

\bibliographystyle{unsrt}
\bibliography{Ref}
\end{document}